\documentclass[a4paper,12pt]{amsart}
\newtheorem{thm}{Theorem}[section]
\newtheorem{defi}[thm]{Definition}
\newtheorem{prop}[thm]{Proposition}
\newtheorem{cor}[thm]{Corollary}
\newtheorem{lemma}[thm]{Lemma}
\theoremstyle{definition}
\newtheorem{rem}[thm]{Remark}

\usepackage{amssymb}
\usepackage{yfonts}
\input xypic
\xyoption{all}

\def\colim{\operatorname{colim}\nolimits}
\def\Comp{\operatorname{Comp}\nolimits}
\def\mComp{\operatorname{\!-Comp}\nolimits}
\def\en{{\mathrm{en}}}
\def\eps{\varepsilon}
\def\mexact{\operatorname{\!-exact}\nolimits}
\def\Ho{\operatorname{Ho}\nolimits}
\def\HH{\operatorname{HH}\nolimits}
\def\Hom{\operatorname{Hom}\nolimits}
\def\Id{\operatorname{Id}\nolimits}
\def\id{\operatorname{id}\nolimits}
\def\im{\operatorname{im}\nolimits}
\def\iso{\buildrel \sim\over\to}
\def\mMOD{\operatorname{\!-Mod}\nolimits}
\def\mMod{\operatorname{\!-mod}\nolimits}
\def\mmodgr{\operatorname{\!-modgr}\nolimits}
\def\mprojgr{\operatorname{\!-projgr}\nolimits}
\def\mproj{\operatorname{\!-proj}\nolimits}
\def\mProj{\operatorname{\!-Proj}\nolimits}
\def\opp{{\operatorname{opp}\nolimits}}
\def\Tor{\operatorname{Tor}\nolimits}
\def\Tot{\operatorname{Tot}\nolimits}

\def\CA{{\mathcal{A}}}
\def\CB{{\mathcal{B}}}
\def\CC{{\mathcal{C}}}
\def\CF{{\mathcal{F}}}
\def\CH{{\mathcal{H}}}
\def\CS{{\mathcal{S}}}
\def\CT{{\mathcal{T}}}
\def\BC{{\mathbf{C}}}
\def\BN{{\mathbf{N}}}
\def\BZ{{\mathbf{Z}}}
\def\GS{{\mathfrak{S}}}

\begin{document}
\title{Khovanov-Rozansky homology and $2$-braid groups}
\author{Rapha\"el Rouquier}
\address{Mathematical Institute,
University of Oxford, 24-29 St Giles', Oxford, OX1 3LB, UK
and Department of Mathematics, UCLA, Box 951555,
Los Angeles, CA 90095-1555, USA}
\email{rouquier@maths.ox.ac.uk}
\maketitle

\section{Introduction}
Khovanov \cite{Kh} has given a construction
of the Khovanov-Rozansky link invariants (categorifying the HOMFLYPT
invariant) using Hochschild cohomology of $2$-braid groups. We give a direct
proof that his construction does give link invariants. We show more
generally that, for any finite Coxeter group, his construction provides a
Markov ``$2$-trace'', and we actually show that the invariant takes value in
suitable derived categories. This makes more precise a result of Trafim Lasy who has shown
that, after taking
the class in $K_0$, this provides a Markov trace \cite{La1,La2}. It coincides with
Gomi's trace \cite{Go} for Weyl groups (Webster and Williamson \cite{WeWi})
as well as for dihedral groups \cite{La1}.

In the first section,
we recall the construction of $2$-braid groups \cite{Rou1}, based on
complexes of Soergel bimodules. The second section is devoted to Markov traces, and
a category-valued version, $2$-Markov traces. We provide a construction using
Hochschild cohomology. The third section is devoted to the proof of the Markov
property for Hochschild cohomology.


\section{Notations}
Let $k$ be a commutative ring. We write $\otimes$ for $\otimes_k$.
Let $A$ be a $k$-algebra. 
We denote by $A^\opp$ the opposite algebra to $A$ and we
put $A^{\en}=A\otimes A^\opp$.

We denote by $A\mMOD$ the category of $A$-modules, by
$A\mMod$ the category of finitely generated $A$-modules, by
$A\mProj$ the category of projective $A$-modules and by 
$A\mproj$ the category of finitely generated projective $A$-modules.
Assume $A$ is graded. We denote by $A\mmodgr$ (resp. $A\mprojgr$)
the category of finitely generated (resp. and projective)
graded $A$-modules.

Given $M$ a graded $k$-module and $n\in\BZ$, we denote by $M\langle n\rangle$
the graded $k$-module given by $M\langle n\rangle_i=M_{n+i}$.

Given $\CA$ an additive category, we denote by $\Comp(\CA)$
(resp. $\Ho(\CA)$) the category
(resp. the homotopy category)
of complexes of objects of $\CA$. If $\CA$ is an abelian category, we
denote by $D(\CA)$ its derived category.

Given $\CC$ a category, we denote by $\{1\}$ the self equivalence of
$\CC^\BZ$ given by $(M\{1\})_i=M_{i+1}$.

Let $\CT$ be a triangulated category equipped with an automorphism $M\mapsto M\langle 1
\rangle$. We denote by $q$ the automorphism of $K_0(\CT)$ given by 
$[M]\mapsto [M\langle 1\rangle]$. This endows $K_0(\CT)$ with a structure
of $\BZ[q,q^{-1}]$-module.
\section{$2$-braid groups}
\subsection{Braid groups}

Let $(W,S)$ be a finite Coxeter group. Let
$V$ be the geometric representation of $W$ over $k=\BC$:
it comes with a basis $\{e_s\}_{s\in S}$.
Given $s\in S$, we denote by $\alpha_s$ the linear form on $V$
such that $s(x)-x=\alpha_s(x)e_s$ for all $x\in V$. The set
$\{\alpha_s\}_{s\in S}$ is a basis of $V^*$. Let $P=P_S=P_{(W,S)}=k[V]$ (we will denote
by $X_S$ or $X_{(W,S)}$ a given object constructed from $(W,S)$).

\smallskip
The braid group $B_S=B_{(W,S)}$ associated to $(W,S)$ is the group generated by
$\{\sigma_s\}_{s\in S}$ with relations
$$\underbrace{\sigma_s\sigma_t\sigma_s\cdots}_{m_{st} \text{ terms}}=
\underbrace{\sigma_t\sigma_s\sigma_t\cdots}_{m_{st} \text{ terms}}$$
for any $s,t\in S$ such that the order $m_{st}$ of $st$ is finite.

We denote by $l:B_S\to\BZ$ the length function. It is the
morphism of groups defined by $l(\sigma_s)=1$ for $s\in S$.

\subsection{Lift}

Let us recall, following \cite{Rou1}, how to lift in a non-trivial
way the action of $W$ on the derived category
$D(P)$ to an action of $B_S$ on the homotopy category $\Ho(P)$.

Let $s\in S$. We put
$$\theta_s=P\otimes_{P^s}P\text{ and }
F_s=0\to \theta_s\langle 1\rangle\xrightarrow{m}P\langle 1\rangle\to 0.$$
The latter
is a complex of graded $P^{\en}$-modules, where $P\langle 1\rangle$ is in cohomological
degree $1$ and $m$ denotes the multiplication map.
We put 
$$F_s^{-1}=0\to P\langle -1\rangle\xrightarrow{a\mapsto a\alpha_s\otimes 1+
a\otimes\alpha_s} \theta_s\to 0.$$
This is a complex of graded $P^{\en}$-modules, where $P\langle -1\rangle$
is in cohomological degree $-1$.

\smallskip
Let us recall a result of \cite[\S 9]{Rou1}.
Given $i_1,\ldots,i_r$, $j_1,\ldots,j_{r'}\in S$ and
$\delta_1,\ldots,\delta_r$, $\eps_1,\ldots,\eps_{r'}\in\{\pm 1\}$ such that
$\sigma_{i_1}^{\delta_1}\cdots \sigma_{i_r}^{\delta_r}=
\sigma_{j_1}^{\eps_1}\cdots \sigma_{j_{r'}}^{\eps_{r'}}$,
there is a canonical isomorphism in $\Ho(P^\en\mmodgr)$
$$F_{i_1}^{\delta_1}\otimes_{P_n}\cdots\otimes_{P_n} F_{i_r}^{\delta_r}\iso
F_{j_1}^{\eps_1}\otimes_{P_n}\cdots\otimes_{P_n} F_{j_{r'}}^{\eps_{r'}}$$
and these isomorphisms form a transitive system of isomorphisms.

Given $b\in B_S$, we put
$$F_b=\lim_{\substack{i_1,\ldots,i_r\\\eps_1,\ldots,\eps_r
\\b=\sigma_{i_1}^{\eps_1}\cdots\sigma_{i_r}^{\eps_r}}}
F_{i_1}^{\eps_1}\otimes_P\cdots\otimes_PF_{i_r}^{\eps_r}\in \Ho(P^\en\mmodgr).$$

\smallskip
The $2$-braid group $\CB_{(W,S)}$ is the full monoidal subcategory of
$\Ho(P^\en\mmodgr)$ with objects the $F_b$'s, with $b\in B_S$.

\subsection{Parabolic subgroups}
Let $I\subset S$ and let $W_I$ be the subgroup of $W$ generated by $I$. Let $V_I=
\bigoplus_{s\in I}ke_s$ and $P_I=k[V_I]$. 
We have $V=V_I\oplus V^I$, hence  $P=P_I\otimes k[V^I]$.
We deduce also that
$V^*=(V^I)^\perp\oplus (V_I)^\perp$, hence the composition of canonical maps
$(V^I)^\perp\hookrightarrow V^*\twoheadrightarrow (V_I)^*$ is an isomorphism.
We identify $(V^I)^\perp=\bigoplus_{s\in I}k\alpha_s$ and $(V_I)^*$ via this isomorphism.

The compositions of canonical maps
$\bigcap_{s{\not\in}I}\ker\alpha_s\to V\to V/V^I$ and
$V_I\to V\to V/V^I$ are isomorphisms: this provides an
isomorphism $V_I\iso \bigcap_{s{\not\in}I}\ker\alpha_s$. We denote by
$\rho_I:P\twoheadrightarrow P_I$ the morphism given by the composition
$V_I\iso \bigcap_{s{\not\in}I}\ker\alpha_s\hookrightarrow V$.

\smallskip
We have a functor $\gamma_I:P_I^\en\mMOD\to P^\en\mMOD$ sending $M$ to 
$k[V^I]\otimes M$, where $k[V^I]$ is the regular $k[V^I]^\en$-module and
$P^\en$ is decomposed as $P^\en=k[V^I]^\en\otimes P_I^\en$.
We obtain a fully faithful monoidal functor
$$\CB_{W_I}\to\CB_W,\ F\mapsto \gamma_I(F)=k[V^I]\otimes F.$$

\section{Hochschild cohomology and traces}
\subsection{Markov traces and $2$-traces}
\subsubsection{Markov traces}
Let $\CC{ox}$ be the poset of finite Coxeter groups, viewed as a category. The objects are
Coxeter groups $(W,S)$ and $\Hom((W,S),(W',S'))$ is the set of injective
maps $f:S\to S'$ such that $m_{f(s),f(t)}=m_{st}$ for all $s,t\in S$.
Given $s\in S$, we 
denote by $i_s:(W_{S\setminus s},S\setminus s)\to (W,S)$ the inclusion.

Let $\CF$ be a full subposet of $\CC{ox}$ closed below.

\smallskip
Let $\CH_{(W,S)}=\BZ[q^{\pm 1}]B_{(W,S)}/((T_s-1)(T_s+q))_{s\in S}$ be
the Hecke algebra of $(W,S)$.

\begin{defi}
Let $R$ be a $\BZ[t_-,t_+,q^{\pm 1}]$-module.
A Markov trace on $\CF$ is the data of a family of $\BZ[q^{\pm 1}]$-linear maps
$\tau_{(W,S)}:\CH_{(W,S)}\to R$ for $(W,S)\in\CF$ such that
\begin{itemize}
\item $\tau_S(hh')=\tau_S(h'h)$ for $h,h'\in\CH_S$
\item $\tau_S(hT_s^{\pm 1})=t_{\pm}\tau_{S\setminus s}(h)$ for 
all $s\in S$ and $h\in\CH_{S\setminus s}$.
\end{itemize}
\end{defi}

Markov traces, with a possibly more general definition,
 have been studied by Jones and Ocneanu in type $A$ \cite{Jo},
Geck-Lambropoulou
in type $B$ \cite{GeLa}, Geck in type $D$ \cite{Ge}, and
Kihara in type $I_2(n)$ \cite{Ki}. Gomi has provided a general construction of
Markov traces for Weyl groups, using Lusztig's Fourier transform \cite{Go}. More
recently, Lasy has studied Markov traces in relation with Gomi's definition and
Soergel bimodules \cite{La1}.

\subsubsection{Markov $2$-traces}
\begin{defi}
Let $\CC:\CF\to\CC{at}$ be a functor.

A Markov $2$-trace on $\CF$ (relative to $\CC$) is the
data of functors $M_{(W,S)}:\CB_{(W,S)}\to \CC_{(W,S)}$ such that the following holds
\begin{itemize}
\item $M_S(?_1\cdot ?_2)\simeq M_S(?_2\cdot ?_1)$ 
as functors $\CB_S\times\CB_S\to\CC_S$
\item $M_S(\gamma_{S\setminus s}(?)\cdot F_s^{\pm 1})\simeq T_{S,s,\pm}
\CC(i_s) M_{S\setminus s}(?)$
as functors $\CB_{S\setminus s}\to\CC_S$, for some
endofunctors $T_{S,s,\pm}$ of $\CC_S$, for all $s\in S$.
\end{itemize}
\end{defi}

One can ask in addition that the functors $T_{S,s,\pm}$ are invertible. On the other
hand, one can get a more general definition by dropping the functoriality of
$\CC$ and by requiring the existence of functors $D_{S,s,\pm}:
\CC_{S\setminus s}\to\CC_S$ such that 
$M_S(\gamma_{S\setminus s}(?)\cdot F_s^{\pm 1})\simeq D_{S,s,\pm}
M_{S\setminus s}(?)$.

\begin{rem}
Let $\bar{\CC}=\colim\CC$ and assume there are endofunctors
$T_\pm$ of $\bar{\CC}$ which restrict to $T_{S,s,\pm}$ for
any $S$ and $s\in S$.  Replacing $\CC_{(W,S)}$ by $\bar{\CC}$, one
can construct from a Markov $2$-trace another one taking value in the constant
category $\bar{\CC}$, and with fixed endofunctors $T_\pm$.
\end{rem}

The first ``trace'' condition, once formulated in the appropriate
homotopical setting, leads to a universal solution (``abelianization'')
that can be described explicitly \cite{BeNa}. It would be interesting to find a
suitable formulation of the second condition in this setting leading to a universal
solution. It would also be interesting to study functoriality with respect to
cobordism in type $A$, along the lines of \cite{KhTh} and \cite{ElKr}.

\subsubsection{From Markov $2$-traces to Markov traces}
\label{se:2totrace}
Let $\CS{oe}$ be the category of Soergel bimodules: this is
the full subcategory of $P^\en\mmodgr$ whose objects
are direct summands of direct sums of objects of the form
$\theta_{s_1}\cdots\theta_{s_n}\langle r\rangle$, for some
$s_1,\ldots,s_n\in S$ and $r\in\BZ$.

 There is a $\BZ[q^{\pm 1}]$-algebra
morphism $\CH\to K_0(\CS{oe})$ given by $T_s\mapsto [F_s]$, and that morphism
is actually an isomorphism (cf \cite[Theorem 1]{Soe} and \cite[Th\'eor\`eme
2.4]{Li}).

\smallskip
We consider now a Markov $2$-trace in the following setting.
Assume the functor $\CC$ takes values in graded triangulated categories and
$M_{(W,S)}$ is the restriction of a graded triangulated functor
$M_{(W,S)}:\Ho^b(\CS{oe}_{(W,S)})\to \CC_{(W,S)}$. In particular,
it induces a $\BZ[q^{\pm 1}]$-linear map $\CH_{(W,S)}\to 
K_0(\CC_{(W,S)})$.
Let $R=\colim_{(W,S)\in\CF}K_0(\CC_{(W,S)})$, a $\BZ[q^{\pm 1}]$-module.
Assume there are commuting endomorphisms $t_\pm$ of $R$ compatible with the
action of $[T_{S,s,\pm}]$ on $K_0(\CC_{(W,S)})$, for all $(W,S)\in\CF$ and $s\in S$,
via the canonical maps $\iota_{(W,S)}:K_0(\CC_{(W,S)})\to R$.

\smallskip
Define $\tau_{(W,S)}:B_{(W,S)}\to R$ by $\tau(b)=\iota_S([M_S(F_b)])$. We have
the following immediate proposition.

\begin{prop}
The maps $\tau_{(W,S)}$ come uniquely from $\BZ[q^{\pm 1}]$-linear
maps $\CH_{(W,S)}\to R$. They define a Markov trace on $\CF$.
\end{prop}

\subsection{Hochschild homology}
\subsubsection{Main Theorem}
We put $\HH_i=\HH_i^{(W,S)}=\Tor_i^{P^{\en}}(P,-):P^{\en}\mmodgr\to
P\mmodgr$. This gives rise to functors
$\HH_i:\Ho^b(P^{\en}\mmodgr)\to \Ho^b(P\mmodgr)$ and to a functor
$\HH_*:\Ho^b(P^{\en}\mmodgr)\to \Ho^b(P\mmodgr)^{\BZ}$.

Given $I\subset S$, we have a functor $\rho_I^*:D^b(P_I\mmodgr)^\BZ\to
D^b(P\mmodgr)^\BZ$. This defines a functor from $\CC{ox}$ to graded
triangulated categories $(W,S)\mapsto D^b(P_S\mmodgr)^\BZ$. Our grading here is the one
coming from $P_S\mmodgr$.

The following theorem is a consequence of Theorem \ref{th:traceHH} below.
\begin{thm}
\label{th:homotopy}
The functors $\HH_*^S$ define a Markov $2$-trace on finite Coxeter groups with
value in $\CC_S=D^b(P_S\mmodgr)^\BZ$ and with
$T_{S,s,+}=[-1]\{1\}$ and
$T_{S,s,-}=\Id$.
\end{thm}

Passing to homology, we obtain the following result.
\begin{cor}
\label{cor:Khovanov}
The functors $H^*\HH_*^S$ define a Markov $2$-trace on finite Coxeter groups
with value in $\CC_S=\bigl((P_S\mmodgr)^\BZ\bigr)^{\BZ}$ and with
$T_{S,s,+}=[-1]\{1\}$ and
$T_{S,s,-}=\Id$.
\end{cor}

The construction of \S\ref{se:2totrace} provides a Markov trace, recovering a result
of Lasy \cite{La1,La2}. Note that Webster and Williamson \cite{WeWi} have shown that for
finite Weyl groups, this is actually Gomi's trace, as conjectured by J.~Michel.
That has been shown to hold also in type $I_2(n)$ by Lasy \cite{La1}.

\begin{cor}[Lasy]
The following defines a Markov trace on finite Coxeter groups:
$$\CB_{(W,S)}\ni b\mapsto \sum_{d,i,j} (-1)^j
\dim H^j(\HH_i^S(F_b))_d q^{-d}t^{-i}\in\BZ[q^{\pm 1},t^{\pm 1}]$$
corresponding to $t_+=-t$ and $t_-=1$.
\end{cor}

\subsubsection{Shift adjustment}
By shifting suitably the invariants, we can get rid of the
automorphisms $T_{S,s,\pm}$, but we lose functoriality (it would be interesting to
see if functoriality with respect to an appropriate notion of cobordisms can be
implemented).
In order to do this, we need to use $\frac{1}{2}\BZ$-complexes.

Given $\CA$ an additive category, 
the category of $\frac{1}{2}$-complexes in $\CA$ has
objects $(C^i,d^i)_{i\in\frac{1}{2}\BZ}$ where the differential
has degree $1$, and morphisms are $\frac{1}{2}\BZ$-graded maps commuting with the
differential. Its homotopy category is denoted by $\Ho_{\frac{1}{2}}(\CA)$ and,
when $\CA$ is an abelian category, its derived category by
$D_{\frac{1}{2}}(\CA)$.

\begin{cor}
Given $(W,S)$ a finite Coxeter group and $b\in B_S$,
let 
$$N_S(F_b)=
\HH_{*-\frac{|S|+l(b)}{2}}^S(F_b)\left[\frac{|S|+l(b)}{2}\right]
\in D^b_{\frac12}(P_S\mmodgr)^{\frac{1}{2}\BZ}.$$

\begin{itemize}
\item
We have
$N_S(F_bF_{b'})\simeq N_S(F_{b'}F_b)$ for all $b,b'\in B_S$.
\item
Given $s\in S$ and $b\in B_{S\setminus s}$, we have
$N_S(\gamma_{S\setminus s}(F_b)F_s^{\pm 1})\simeq \rho_{S\setminus s}^*N_{S\setminus s}(F_b)$.
\end{itemize}
\end{cor}

\subsection{Khovanov-Rozansky homology of links}
We specialize now to the case of the classical Artin braid groups considered by
Khovanov in \cite{Kh}.
Note that
Khovanov conjectured ten years ago that the $F_b$'s should give rise to
interesting link invariants.

\smallskip
We take here $V=(\bigoplus_{i=1}^nke_i)/k(e_1+\cdots e_n)$, the reflection
representation of $W=\GS_n$, with $S=\{(1,2),\ldots,(n-1,n)\}$. Let
$P_n=k[V]=k[\alpha_1,\ldots,\alpha_{n-1}]$, where $\alpha_i=X_{i+1}-X_i$.
We put $B_n=B_{(W,S)}$.

\smallskip
Let $P_\infty=\lim_n P_n$, where the limit is taken over the morphisms
of $P_n$-algebras $\rho_n:P_{n+1}\to P_n,\ \alpha_n\mapsto 0$.
This provides
functors between derived categories
$$\cdots\to D^b(P_n\mmodgr)\to D^b(P_{n+1}\mmodgr)\to\cdots
\to D^b(P_\infty\mmodgr).$$

\begin{thm}
\label{th:homotopyBn}
The assignment to $b\in B_{n+1}$ of
the isomorphism class of $\HH_{*-\frac{n+l(b)}{2}}(F_b)[\frac{n+l(b)}{2}]$ in
$D^b_{\frac12}(P_\infty\mmodgr)^{\frac{1}{2}\BZ}$ defines an invariant of oriented links.
\end{thm}

Passing to homology, we recover the following result of Khovanov \cite{Kh}. Khovanov
identifies the invariant as the Khovanov-Rozansky homology, a categorification of
the HOMFLYPT polynomial.

\begin{thm}[Khovanov]
\label{th:Khovanov}
The assignment to $b\in B_{n+1}$ of
$$X_b=(t_2t_3^{-1})^{(n+l(b))/2}
\sum_{d,i,j}
\dim H^j(\HH_i(F_b))_d t_1^dt_2^it_3^j\in\BN[t_1^{\pm 1},t_2^{\pm 1/2},t_3^{\pm
1/2}]$$
defines an invariant of oriented links.
\end{thm}

Note that $X_1=1$, where $1\in B_1$ corresponds to the trivial knot.

\smallskip
We define now a two variables invariant 
$Y_b=(X_b)_{|t_3^{1/2}=\sqrt{-1}}$.
The following corollary shows that $Y_b$ is the HOMFLYPT
polynomial, as expected.

\begin{cor}[Khovanov]
\label{cor:homfly}
Given $b,b'\in B_n$ and $r\in\{1,\ldots,n-1\}$, we have
$$t_1^{-1/2}t_2^{1/2}Y_{b\sigma_r^{-1}b'}+
t_1^{1/2}t_2^{-1/2}Y_{b\sigma_rb'}=\sqrt{-1}(t_1^{-1/2}-t_1^{1/2})Y_{bb'}.$$
\end{cor}

\section{Proofs}
\subsection{Multiple complexes}

\subsubsection{Total objects}
For a more intrinsic approach to this section, cf \cite[\S 1.1]{De}.

Let $\CA$ be an additive category and $n\ge 0$.
We denote by $\Comp(\CA)$ the category of complexes of objects of $\CA$.
The category $n\mComp(\CA)$
of $n$-fold complexes is defined inductively by
$n\mComp(\CA)=\Comp((n-1)\mComp(\CA))$ and $0\mComp(\CA)=\CA$.
Its objects are families $(X,d_1,\ldots,d_n)$ where 
$X$ is an object of $\CA$ graded by $\BZ^n=\bigoplus_{i=1}^n \BZ e_i^*$,
$d_i$ is a graded map of degree $e_i$ and $d_i^2=[d_i,d_j]=0$ for
all $i,j$.

Given $X$ an $n$-complex and $i\in\{1,\ldots,n\}$, we define 
$Y=X[e_i]$ as the $n$-complex given by $Y^b=X^{e_i+b}$ and
differentials $\partial_i^b=(-1)^{\delta_{ij}} d_j^{e_i+b}$.

\medskip
Let $f:\{1,\ldots,n\}\to \{1,\ldots,m\}$ be a map.
It induces a map $\sigma:\BZ^n\to\BZ^m$ and gives by duality
a map $\BZ^m\to\BZ^n$. This provides a functor
from $\BZ^n$-graded objects to $\BZ^m$-graded objects of $\CA$.
Let $X$ be an $n$-complex. We have a corresponding $\BZ^m$-graded
object $X'$. We define a structure of $m$-complex by
$$d_i^{\prime a}=\sum_{\substack{b\in\sigma^{-1}(a)\\ j\in f^{-1}(i)}}
(-1)^{\sum_{k\in f^{-1}(i),k<j}b_k}d_j^b$$
where $b=\sum_i b_i e_i$.

Note that when $f$ is an injection, the sum above has only positive signs.
When $f$ is a bijection, then $\Tot^f$ is a self-equivalence of
$n\mComp(\CA)$. We write $\Tot=\Tot^f$ when $m=1$.

\smallskip
We have defined an additive functor 
$$\Tot^f: n\mComp(\CA)\to m\mComp(\CA).$$

\medskip
Let $g:\{1,\ldots,m\}\to \{1,\ldots,p\}$ be a map and
$\tau:\BZ^m\to\BZ^p$ the associated morphism.
Let $X\in n\mComp(\CA)$. The $\BZ^p$-graded objects
underlying $\Tot^{gf}(X)$ and $\Tot^g(\Tot^f(X))$ have their
component of degree $a$ equal to $\bigoplus_{c\in (\tau\sigma)^{-1}(a)}
X^c$. We
define an isomorphism between these $p$-complexes by multiplication by
$(-1)^{\eps(c)}$ on $X^c$, where
$$\eps(c)=\sum_{\substack{l<l'\\ f(l)>f(l')\\ gf(l)=gf(l')}}
c_l c_{l'}.$$
This gives an isomorphism of functors
$$\Tot^{gf}\iso \Tot^g\circ \Tot^f.$$

\medskip
Let $k$ be a commutative ring and $A$, $B$ and $C$ be three $k$-algebras.
Let $X\in n\mComp((A\otimes B)\mMOD)$ and
$Y\in m\mComp((B\otimes C)\mMOD)$. Then
$X\otimes_B Y$ defines an object of $(n+m)\mComp((A\otimes C)\mMOD)$.
We have $(X\otimes_B Y)^{(a_1,\ldots,a_{n+m})}=
X^{(a_1,\ldots,a_n)}\otimes_B Y^{(a_{n+1},\ldots,a_{n+m})}$.

\subsubsection{Cohomology}
Assume $\CA$ is an abelian category.
Let $X\in n\mComp(\CA)$. Let $r\in\{1,\ldots,n\}$ and
$Y=\ker d_r/\im d_r$. This is an $n$-complex with $d_{r,Y}=0$.
Let $i\in\BZ$.
Consider the map $f:\BZ^{n-1}\to\BZ^n,\ (a_1,\ldots,a_{n-1})\mapsto
(a_1,\ldots,a_{r-1},i,a_r,\ldots,a_{n-1})$.
We put $H^i_{d_r}(X)=\bigoplus_{a\in\BZ^{n-1}} Y^{f(a)}$. This
defines a functor 
$$H^i_r:n\mComp(\CA)\to (n-1)\mComp(\CA).$$

\smallskip
Let $g:\{1,\ldots,n\}\to\{1,\ldots,m\}$ be a map and let $r\in\{1,\ldots,m\}$
such that $f^{-1}(r)=\{s\}$ for some $s\in\{1,\ldots,n\}$.
Define maps
$$\alpha:\{1,\ldots,n-1\}\to \{1,\ldots,n\},\ i\mapsto \begin{cases}
i & \text{ if }i<s\\
i+1 & \text{ if }i\ge s
\end{cases}$$
and
$$\beta:\{1,\ldots,m\}\to \{1,\ldots,m-1\},\ i\mapsto \begin{cases}
i & \text{ if }i<r\\
i-1 & \text{ if }i\ge r
\end{cases}$$
Then, we have a canonical isomorphism
\begin{equation}
\label{eq:HTot}
H^i_{d_r}(\Tot^f(M))\iso \Tot^{\beta f\alpha}(H^i_{d_s}(M)).
\end{equation}

\subsubsection{Resolutions}
Let $k$ be a commutative ring and $A$ a $k$-algebra.
The $k$-linear
functor $H^0:\Ho^-(A\mProj)\to A\mMOD$ restricted to
the full subcategory of complexes $M$ with $H^i(M)=0$ for $i\not=0$ is
an equivalence. Let $C=C_A:A\mMOD\to \Ho^-(A\mProj)$ be an inverse, composed
with the inclusion functor. By construction the resolution functor $C$ is
fully faithful.
It induces a functor, still denoted by $C$, 
$$C:\Ho^b(A\mMOD)\to\Ho^b\bigl(\Ho^-(A\mProj)\bigr).$$
We view $\Ho^b\bigl(\Ho^-(A\mProj)\bigr)$ as a triangulated category with the canonical
structure on $\Ho^b(\CC)$, where $\CC$ is the additive category
$\Ho^-(A\mProj)$. The functor $C$ is a fully faithful triangulated functor.

\smallskip
Assume $A$ is projective as a $k$-module.
Let $X=C_{A^\en}(A)$ be a projective resolution of $A$ as an $A^\en$-module.
The functor
$$-\otimes_{A^\opp}X:\mathrm{Comp}^b(A\mMOD)\to 2\mComp(A\mProj)$$
composed with the canonical functor $2\mComp(A\mProj)\to
\Ho\bigl(\Ho(A\mProj)\bigr)$ is isomorphic to $C$: we have
$M\otimes_{A^\opp}X\iso C_A(M)$ for $M\in\Ho^b(A\mMOD)$.

\medskip
Let $i\in\BZ$. The functor $H^i:\Ho^-(A\mProj)\to A\mMOD$ induces a functor
$$H^i_{d_2}:\Ho^b\bigl(\Ho^-(A\mProj)\bigr)\to \Ho^b(A\mMOD).$$
It extends
the functor $H^i_{d_2}:2\mComp(A\mMOD)\to \mathrm{Comp}(A\mMOD)$.

\medskip
Let $B,B'$ be two $k$-algebras, projective as $k$-modules.
Let $L\in\Ho^b\bigl((A\otimes B^\opp)\mMOD\bigr)$ and
$M\in\Ho^b((B\otimes B^{\prime\opp})\mMOD)$. Assume the components of $M$ are projective
right $B'$-modules. We deduce from (\ref{eq:HTot}) an isomorphism
\begin{equation}
\label{eq:resoltensor}
\Tot^{13\to 1,2\to 2}\bigl(C_{A\otimes B^\opp}(L)\otimes_B M\bigr)\iso
C_{A\otimes B^{\prime\opp}}\bigl(\Tot(L\otimes_B M)\bigr).
\end{equation}

\smallskip
Note that when $A$ is coherent, then $A\mMOD$ can be replaced by the
abelian category $A\mMod$ and $A\mProj$ by $A\mproj$. If $A$ is graded, we
can replace these categories by the corresponding categories of graded modules.

\subsection{Markov moves}
\label{se:proofMarkov}
Let $(W,S)$ be a finite Coxeter group.
Let $P\mexact$ be the category of finitely generated
graded $P^\en$-modules whose restrictions to $P$ and $P^\opp$ are projective.
Let $M\in \Ho^b(P\mexact)$ and $i\in\BZ$.
We put
$$K^i(M)=K^i_S(M)=H^i_{d_2}\bigl(P\otimes_{P^\en}C_{P^\en}(M)\bigr)\in
\Ho^b(P\mmodgr).$$
This defines a triangulated functor
$$K^i:\Ho^b(P\mexact)\to \Ho^b(P\mmodgr).$$

\begin{thm}
\label{th:traceHH}
Given $N,N'\in\Ho^b(P\mexact)$, we have functorial isomorphisms
$K^i(N\otimes_{P} N')\simeq K^i(N'\otimes_P N)$ in $\Ho(P\mmodgr)$.

Let $s\in S$ and let $z_s$ be a non-zero element of $(V^*)^{S\setminus s}$.
Let $M\in\Ho^b(P_{S\setminus s}\mexact)$.
We have functorial isomorphisms
\begin{itemize}
\item
$K^i_S\bigl(\gamma_{S\setminus s}(M)\otimes_{P}F_s\bigr)\simeq
\rho_{S\setminus s}^* K^{i+1}_{S\setminus s}(M)[-1] \text{ in } D(P\mmodgr)$
\item
$K^i_S\bigl(\gamma_{S\setminus s}(M)\otimes_{P}F_s^{-1}\bigr)\simeq
\rho_{S\setminus s}^* K^i_{S\setminus s}(M)\text{ in } D(P\mmodgr)$
\item 
$K^i_S\bigl(\gamma_{S\setminus s}(M))\simeq
P\otimes_{P_{S\setminus s}} K^{i+1}_{S\setminus s}(M)\langle -1\rangle\oplus
P\otimes_{P_{S\setminus s}} K^i_{S\setminus s}(M)
 \text{ in } \Ho(P\mmodgr)$.
\end{itemize}
\end{thm}

\begin{proof}
Thanks to (\ref{eq:resoltensor}), we have
$$C_{P^\en}\bigl(\Tot(N\otimes_PN')\bigr)\simeq
\Tot^{12\to 1,3\to 2}(C_{P^\en}(N)\otimes_P N')$$
hence
\begin{align*}
P\otimes_{P^\en}C_{P^\en}\bigl(\Tot(N\otimes_PN')\bigr)&\simeq
\Tot^{12\to 1,3\to 2}\bigl(C_{P^\en}(N)\otimes_{P^\en}N'\bigr)\\
&\simeq \Tot^{12\to 1,3\to 2}\bigl(N'\otimes_{P^\en}C_{P^\en}(N)\bigr)\\
&\simeq P\otimes_{P^\en}C_{P^\en}\bigl(\Tot(N'\otimes_PN)\bigr)
\end{align*}
and the first assertion follows.

\medskip
We have
$$C_{P^\en}(M\otimes k[z_s])\simeq
\Tot^{1\to 1, 23\to 2}\bigl(C_{P_{S\setminus s}^\en}(M)\otimes X\bigr)$$
where
$X=0\to k[z_s]^\en\langle -1\rangle \xrightarrow{z_s\otimes 1
-1\otimes z_s} k[z_s]^\en \to 0$, the non-zero terms being
in degrees $-1$ and $0$.
Let 
$$L=P\otimes_{P^\en}\biggl(C_{P^\en}\Bigl(
\Tot\bigl((M\otimes k[z_s])
\otimes_{P}F_s\bigr)\Bigr)\biggr).$$
By (\ref{eq:resoltensor}), we have
$$L\simeq
P\otimes_{P^\en}\Tot^{13\to 1,2\to 2}
\bigl(C_{P^\en}(M\otimes k[z_s]) \otimes_{P}F_s\bigr)
\simeq \Tot^{13\to 1,24\to 2}\bigl(
(F_s\otimes_{k[z_s]^\en}X)
\otimes_{P_{S\setminus s}^\en} C_{P_{S\setminus s}^\en}(M)\bigr)$$
We have
$$F_s\otimes_{k[z_s]^\en}X\simeq
\xy (0,8) *+{\theta_s}="1",
(20,8) *+{P}="2",
(0,-8) *+{\theta_s\langle 1\rangle}="3",
(20,-8) *+{P\langle 1\rangle}="4",
\ar^-m"1";"2",
\ar_{z_s\otimes 1-1\otimes z_s}"1";"3",
\ar^0"2";"4",
\ar_-m"3";"4"
\endxy
\simeq
\xy (0,8) *+{\theta_s}="1",
(20,8) *+{P}="2",
(0,-8) *+{\theta_s\langle 1\rangle}="3",
(20,-8) *+{P\langle 1\rangle}="4",
\ar^-m"1";"2",
\ar_{\alpha_s\otimes 1-1\otimes \alpha_s}"1";"3",
\ar^0"2";"4",
\ar_-m"3";"4"
\endxy
$$
Indeed, there is $c\in k^*$ such that $s(z_s)-z_s=-2c\alpha_s$. Then
$z_s-c\alpha_s\in (V^*)^s$, hence $z_s\otimes 1-1\otimes z_s$ and
$c(\alpha_s\otimes 1-1\otimes \alpha_s)$ are equal in $\theta_s$.

Lemma \ref{le:splittingF} below shows that, when $s{\not\in}Z(W)$, then
the exact sequence of $P^\en$-modules
$$0\to P\xrightarrow{\alpha_s\otimes 1+1\otimes \alpha_s}
\theta_s\xrightarrow{a\otimes b\mapsto as(b)} Ps\to 0$$
splits by restriction to $P_{S\setminus s}^\en$. Here,
$Ps=P$ as a left $P$-module, and the right action of $a\in P$ is given by
multiplication by $s(a)$. Note that when $s\in Z(W)$,
then $P^s=P_{S\setminus s}\otimes k[\alpha_s^2]$, and the splitting
of the sequence holds trivially.

We deduce that there is an isomorphism of complexes of graded
$P_{S\setminus s}^\en$-modules
$$\bigl(0\to\theta_s\xrightarrow{\alpha_s\otimes 1-1\otimes
\alpha_s}
\theta_s\langle 1\rangle\to 0\bigr)\simeq P[1]\langle -1\rangle \oplus
\bigl(0\to Ps\xrightarrow{a\mapsto a\alpha_s\otimes 1-a\otimes \alpha_s}
\theta_s \to 0\bigr)\langle 1\rangle.$$

$$\text{Let }Y_1=
\xy (0,8) *+{P\langle-1\rangle}="1",
(20,8) *+{P}="2",
(0,-8) *+{0}="3",
(20,-8) *+{0}="4",
\ar^-{2\alpha_s}"1";"2",
\ar"1";"3",
\ar"2";"4",
\ar"3";"4"
\endxy
\text{ and }
Y_2=
\xy (0,8) *+{Ps}="1",
(20,8) *+{0}="2",
(0,-8) *+{\theta_s\langle 1\rangle}="3",
(20,-8) *+{P\langle 1\rangle}="4",
\ar"1";"2",
\ar_{a\mapsto a\alpha_s\otimes 1-a\otimes \alpha_s}"1";"3",
\ar"2";"4",
\ar"3";"4"
\endxy$$
We have an exact sequence of bicomplexes of graded
$P^\en$-modules
$$0\to Y_1 \to F_s\otimes_{k[\alpha_s]^\en}X \to Y_2\to 0.$$
It splits after restricting to $P_{S\setminus s}^\en$ and
applying $?^{(i,*)}$. It follows that
we have an exact sequence of complexes of graded $P^\en$-modules
\begin{multline*}
0\to H^i_{d_2}\Bigl(\Tot^{13\to 1,24\to 2}
\bigl(Y_1\otimes_{P_{S\setminus s}^\en}C_{P_{S\setminus s}^\en}(M)\bigr)\Bigr)
\to
H^i_{d_2}(L) \to \\
\to H^i_{d_2}\Bigl(\Tot^{13\to 1,24\to 2}
\bigl(Y_2\otimes_{P_{S\setminus s}^\en}C_{P_{S\setminus s}^\en}(M)\bigr)\Bigr)
\to 0.
\end{multline*}

\smallskip
We have an exact sequence of $P^\en$-modules
$$0\to Ps \xrightarrow{a\mapsto a\alpha_s\otimes 1-a\otimes \alpha_s}
\theta_s \xrightarrow{m} P\to 0.$$
So,
the morphism of bicomplexes $Y_2\to Y'_2$:
$$\xymatrix{
Ps\ar[r] 
\ar[d]_{a\mapsto a\alpha_s\otimes 1-a\otimes \alpha_s} &
 0\ar[d] &&
 0 \ar[d]\ar[r] & 0\ar[d] \\
 \theta_s\langle 1\rangle \ar[r]_-m \ar@/_2pc/[rrr]_m & P\langle 1\rangle
 \ar@/_2pc/[rrr]_{\id}&&
P\langle 1\rangle \ar[r]^-{\id} & P\langle 1\rangle}
$$
induces an isomorphism of complexes
$$H^i_{d_2}\Bigl(\Tot^{13\to 1,24\to 2}
\bigl(Y_2\otimes_{P_{S\setminus s}^\en}C_{P_{S\setminus s}^\en}(M)\bigr)\Bigr)
\iso
H^i_{d_2}\Bigl(\Tot^{13\to 1,24\to 2}\bigl(Y'_2\otimes_{P_{S\setminus s}^\en}
C_{P_{S\setminus s}^\en}(M)\bigr)\Bigr)$$
and these complexes vanish in $\Ho^b(P^\en\mmodgr)$.

We deduce that
\begin{align*}
H^i_{d_2}(L) &\simeq
H^i_{d_2}\Bigl(\Tot^{13\to 1,24\to 2}
\bigl(Y_1\otimes_{P_{S\setminus s}^\en}C_{P_{S\setminus s}^\en}(M)\bigr)\Bigr)\\
&\simeq
\rho_{S\setminus s}^*
H^i_{d_2}\Bigl(\Tot^{13\to 1,24\to 2}
\bigl(P_{S\setminus s}[(-1,1)]
\otimes_{P_{S\setminus s}^\en}C_{P_{S\setminus s}^\en}(M)\bigr)\Bigr)\\
&\simeq \rho_{S\setminus s}^*
H^{i+1}_{d_2}\bigl(P_{S\setminus s}\otimes_{P_{S\setminus s}^\en}
C_{P_{S\setminus s}^\en}(M)\bigr)[-1]
\end{align*}
in $D^b(P^\en\mmodgr)$. Note that the multiplication map
$P^\en\to P$ is a split surjection of algebras. We deduce the first and
last terms of the sequence of isomorphisms above are actually isomorphic in
$D^b(P\mmodgr)$.
This shows the second assertion.
The proof of the assertion involving $F_s^{-1}$ is similar.

\smallskip
We have
$k[z_s]\otimes_{k[z_s]^\en}X\simeq k[z_s]\langle -1\rangle[1]\oplus k[z_s]$ and the
last assertion follows immediately.
\end{proof}

\begin{lemma}
\label{le:splittingF}
Let $s\in S$. Assume $s{\not\in}Z(W)$.
Let $L=(V^{S\setminus s})^\perp \cap (V^*)^s$, a hyperplane
of $(V^{S\setminus s})^\perp=(V_{S\setminus s})^*$.
We have a commutative diagram of $P_{S\setminus s}^\en$-modules where the
diagonal map is an isomorphism
$$\xymatrix{
\theta_s\ar[r]^-{a\otimes b\mapsto as(b)} & Ps \\
P_{S\setminus s}\otimes_{S(L)}P_{S\setminus s}
\ar[u]^{a\otimes b\mapsto a\otimes b}
 \ar[ur]_{a\otimes b\mapsto as(b)}^\sim}$$
Here, $Ps$ denotes the left $P_{S\setminus s}$-module $P$ endowed
with a right action of $a\in P_{S\setminus s}$ by multiplication by $s(a)$.
\end{lemma}

\begin{proof}
We will show that $\phi:P_{S\setminus s}\otimes_{S(L)}P_{S\setminus s}\to
P,\ a\otimes b\mapsto as(b)$ is an isomorphism. It is a morphism
of algebras, and a morphism of graded left $P_{S\setminus s}$-modules.

By assumption, there is $t\in S\setminus s$ such that $m_{st}\not=2$, so that
$s(\alpha_t)-\alpha_t$ is a non-zero multiple of $\alpha_s$. It follows that
$\phi(\alpha_t\otimes 1-1\otimes \alpha_t)\in k^\times\alpha_s$. 
Since $V^*=(V^{S\setminus s})^\perp\oplus k\alpha_s$, we have
$P=P_{S\setminus s}\otimes k[\alpha_s]$ and we deduce that $\phi$ is
surjective. 

Since $\phi$ is a morphism of graded free left $P_{S\setminus s}$-modules
with the same graded ranks $1+t+t^2+\cdots$, we deduce that $\phi$ is an
isomorphism.
\end{proof}

\begin{rem}
While we do not expect the category of Soergel bimodules to exist for
complex reflection groups, we hope that its homotopy category does exist,
as well as the $2$-braid group. This would be a starting point for a
structural approach to the construction of unipotent data in
Brou\'e--Malle--Michel's theory of spets
\cite{BroMaMi}.
\end{rem}

\end{document}